\documentclass[11pt]{article}
\usepackage{amsfonts,amscd,amssymb, amsmath}
\newtheorem{definition}{Definition}[section]

\newtheorem{remark}[definition]{Remark}

\newtheorem{theorem}[definition]{Theorem}
\newtheorem{example}[definition]{Example}

\newcommand{\nat}{\mbox{$\;\natural \;$}}

\def\rawo\lonra{\longrightarrow}

\def\ot{\otimes}

\newcommand{\selabel}[1]{\label{se:#1}}

\allowdisplaybreaks[4]

\newenvironment{proof}{{\it Proof.}}{\hfill $ \square $ \vskip 4mm}

\begin{document}
\title{L-R-crossed products}
\author {Florin Panaite\\
Institute of Mathematics of the 
Romanian Academy\\ 
PO-Box 1-764, RO-014700 Bucharest, Romania\\
e-mail: Florin.Panaite@imar.ro
}
\date{}
\maketitle

\begin{abstract}
Given an associative algebra $H$, a linear space $U$ and some linear maps $J$, $T$, $\gamma $, $\eta $ satisfying some axioms, 
we define an associative algebra structure on $U\otimes H$, called an L-R-crossed product. This contains as particular cases some 
previous constructions, such as the (iterated) Brzezi\'{n}ski crossed product and the L-R-smash 
product over quasi-bialgebras. \\
\textbf{Keywords}: twisted tensor product, L-R-smash product, Brzezi\'{n}ski crossed product \\
\textbf{MSC2020}: 16S35, 16T99
\end{abstract}
\section*{Introduction}
${\;\;\;\;}$
There exist in the literature several constructions of the following type: given an associative algebra $H$ and a linear space $U$, together with some extra data, an associative algebra structure on $U\otimes H$ (or on $H\otimes U$) is built. Clearly, the easiest example is the usual tensor product of associative algebras. A generalization of it, called twisted tensor product of algebras, 
was introduced in \cite{Cap}, \cite{VanDaele}: one starts with two associative algebras $A$ and $B$ and a linear map $R:B\otimes A\rightarrow A\otimes B$ satisfying certain axioms and one obtains an associative algebra structure on $A\otimes B$, denoted by 
$A\otimes _RB$. Twisted tensor products of algebras have various applications and have been studied intensively in the last decades, see for instance \cite{gustafson}, \cite{jlpvo}, \cite{ocal}
and references therein. A slight generalization of twisted tensor products of algebras was introduced in \cite{ciungu} under the name 
L-R-twisted tensor product of algebras (it involves an extra map $Q:A\otimes B\rightarrow A\otimes B$). A substantial 
generalization is the so called Brzezi\'{n}ski crossed product introduced in \cite{brz} (it contains as well as particular case the Hopf crossed product). In a Brzezi\'{n}ski crossed product only one of the tensor factors is an associative algebra, the other one is just a linear space. One can introduce as well a mirror version of the Brzezi\'{n}ski crossed product, and it turns out that under certain circumstances Brzezi\'{n}ski crossed products can be iterated, see \cite{iterated}. 

Another example of obtaining an associative algebra structure by tensoring an associative algebra with a linear space not carrying an associative algebra structure is the so-called L-R-smash product $\mathcal{A}\nat H$, where $H$ is a quasi-bialgebra and 
$\mathcal{A}$ is an $H$-bimodule algebra, introduced in \cite{pvo}. When $H$ is moreover a quasi-Hopf algebra, $\mathcal{A}\nat H$ is isomorphic to a construction previously introduced in \cite{hn}, \cite{bpvo} under the name diagonal crossed product, denoted by 
$\mathcal{A}\bowtie H$ (this can be defined only for quasi-Hopf algebras, not for quasi-bialgebras); we recall that if $H$ is finite 
dimensional, $H^*\bowtie H$ is the algebra structure of the quantum double of $H$. 

The aim of this paper is to introduce a general construction, called L-R-crossed product, as follows. Given an associative algebra $H$, a linear space $U$ and linear maps $J:H\ot U\rightarrow U\ot H$, 
$T:U\ot H\rightarrow U\ot H$, 
$\gamma : U\ot U\rightarrow U\ot U\ot H$, 
$\eta : U\ot U\rightarrow U\ot H\ot H$ satisfying the axioms (\ref{lrc1})-(\ref{lrc15}) below, we define an associative 
algebra structure on $U\otimes H$, denoted by $U\nat _{J, T, \gamma , \eta }\;H$. It turns out that all the constructions mentioned above are particular cases of this new construction, which can thus be regarded as a common generalization and unification of the previous constructions. 

Since our construction generalizes both L-R-smash products and Brzezi\'{n}ski crossed products, we believe that the name 
L-R-crossed product is justified. However, some terminological confusion might arise, because there exists another construction called L-R-crossed product, introduced in \cite{chenwang} (it contains as particular cases the Hopf crossed product and the L-R-smash product over Hopf algebras, but not over quasi-bialgebras or quasi-Hopf algebras). It turns out that our L-R-crossed product and the one in \cite{chenwang} are independent constructions, none of them is a particular case of the other (although our L-R-crossed product is not far from containing the one in \cite{chenwang} as a particular case). Whether one can define an even more general construction, containing both L-R-crossed products as particular cases, remains an open question. 

\section{The definition of the L-R-crossed product}\selabel{2}
\setcounter{equation}{0}
${\;\;\;\;}$
We work over a commutative field $k$. All algebras, linear spaces
etc. will be over $k$; unadorned $\ot $ means $\ot_k$. By ''algebra'' we 
always mean an associative unital algebra. The multiplication 
of an algebra $A$ is denoted by $\mu _A$ or simply $\mu $ when 
there is no danger of confusion, and we usually denote 
$\mu _A(a\ot a')=aa'$ for all $a, a'\in A$; the unit of $A$ is denoted by $1_A$.  

\begin{theorem}
Let $(H, \mu _H, 1_H)$ be an (associative unital) algebra and $U$ a vector space equipped with a 
distinguished nonzero element $1_U\in U$. Assume that we are given linear maps (with respective notation)
\begin{eqnarray*}
&&J:H\ot U\rightarrow U\ot H, \;\;\;J(h\ot u)=u_J\ot h_J, \\
&&T:U\ot H\rightarrow U\ot H, \;\;\;T(u\ot h)=u_T\ot h_T, \\
&&\gamma : U\ot U\rightarrow U\ot U\ot H, \;\;\;\gamma (u\ot u')=\gamma _1(u, u')\ot 
\gamma _2(u, u')\ot \gamma _3(u, u'), \\
&&\eta : U\ot U\rightarrow U\ot H\ot H, \;\;\;\eta (u\ot u')=\eta _1(u, u')\ot 
\eta _2(u, u')\ot \eta _3(u, u').
\end{eqnarray*}
Assume that the following conditions are satisfied, for all $u, u', u''\in U$ and $h, h'\in H$ (we denote 
by $j$ and $t$ some copies of $J$ and respectively $T$):
\begin{eqnarray}
&&J(h\ot 1_U)=1_U\ot h, \;\;\;J(1_H\ot u)=u\ot 1_H, \label{lrc1} \\[2mm]
&&T(u\ot 1_H)=u\ot 1_H, \;\;\;T(1_U\ot h)=1_U\ot h, \label{lrc2}\\[2mm]
&&\gamma (u\ot 1_U)=u\ot 1_U\ot 1_H, \;\;\;\gamma (1_U\ot u)=1_U\ot u\ot 1_H, \label{lrc3} \\[2mm]
&&\eta (u\ot 1_U)=\eta (1_U\ot u)=u\ot 1_H\ot 1_H, \label{lrc4}\\[2mm]
&&u_J\ot (hh')_J=(u_J)_j\ot h_jh'_J, \label{lrc5} \\[2mm]
&&u_T\ot (hh')_T=(u_T)_t\ot h_Th'_t, \label{lrc6}\\[2mm]
&&\eta _1(u_J, u'_j)\ot \eta _2(u_J, u'_j)(h_J)_j\ot \eta _3(u_J, u'_j)=
\eta _1(u, u')_J\ot h_J\eta _2(u, u')\ot \eta _3(u, u'), \label{lrc9}\\[2mm]
&&\eta _1(u_t, u'_T)\ot \eta _2(u_t, u'_T)\ot (h_T)_t\eta _3(u_t, u'_T)=
\eta _1(u, u')_T\ot \eta _2(u, u')\ot \eta _3(u, u')h_T, \label{lrc10}\\[2mm]
&&\eta _1(\eta _1(u, u'), u''_J)\ot \eta _2(\eta _1(u, u'), u''_J)\eta _2(u, u')_J\ot 
\eta _3(u, u')\eta _3(\eta _1(u, u'), u''_J)\nonumber \\[2mm]
&&\;\;\;\;=\eta _1(u_T, \eta _1(u', u''))\ot \eta _2(u_T, \eta _1(u', u''))\eta _2(u', u'')\nonumber \\[2mm]
&&\;\;\;\;\;\;\;\;\;\ot 
\eta _3(u', u'')_T\eta _3(u_T, \eta _1(u', u'')), \label{lrc11}\\[2mm]
&&\gamma _1(u, u'_J)_T\ot \gamma _2(u, u'_J)\ot \gamma _3(u, u'_J)(h_J)_T\nonumber \\[2mm]
&&\;\;\;\;=\gamma _1(u_T, u')\ot \gamma _2(u_T, u')_J\ot (h_T)_J\gamma _3(u_T, u'), \label{lrc12}\\[2mm]
&&\gamma _1(u, \eta _1(u', u''))_T\ot \gamma _2(u, \eta _1(u', u''))\ot 
\gamma _3(u, \eta _1(u', u''))\eta _2(u', u'')_T\ot \eta _3(u', u'')\nonumber \\[2mm]
&&\;\;\;\;=\gamma _1(\gamma _1(u, u'), u'')\ot \eta _1(\gamma _2(u, u'), 
\gamma _2(\gamma _1(u, u'), u'')_J)\nonumber \\[2mm]
&&\;\;\;\;\;\;\;\;\;\ot \eta _2(\gamma _2(u, u'), \gamma _2(\gamma _1(u, u'), u'')_J)\gamma _3(u, u')_J
\gamma _3(\gamma _1(u, u'), u'')\nonumber \\[2mm]
&&\;\;\;\;\;\;\;\;\;\;\;\;\;\;\ot  \eta _3(\gamma _2(u, u'), \gamma _2(\gamma _1(u, u'), u'')_J), \label{lrc13}\\[2mm]
&&\gamma _1(\eta _1(u, u'), u'')\ot \gamma _2(\eta _1(u, u'), u'')_J\ot \eta _2(u, u')\ot 
\eta _3(u, u')_J\gamma _3(\eta _1(u, u'), u'')\nonumber \\[2mm]
&&\;\;\;\;=\eta _1(\gamma _1(u, \gamma _2(u', u''))_T, \gamma _1(u', u''))
\ot \gamma _2(u, \gamma _2(u', u''))\nonumber \\[2mm]
&&\;\;\;\;\;\;\;\;\ot \eta _2(\gamma _1(u, \gamma _2(u', u''))_T, \gamma _1(u', u''))\nonumber \\[2mm]
&&\;\;\;\;\;\;\;\;\ot 
\gamma _3(u, \gamma _2(u', u''))\gamma _3(u', u'')_T\eta _3(\gamma _1(u, \gamma _2(u', u''))_T, \gamma _1(u', u'')),
\label{lrc14} \\[2mm]
&&\gamma _1(u, \gamma _1(u', u'')_T)\ot \gamma _2(u, \gamma _1(u', u'')_T)_J\ot 
h'_J\gamma _3(u, \gamma _1(u', u'')_T)\nonumber \\[2mm]
&&\;\;\;\;\;\;\;\;\;\;\;\;\;\;\;\;\ot \gamma _2(u', u'')\ot \gamma _3(u', u'')h_T\nonumber \\[2mm]
&&\;\;\;\;=\gamma _1(u, u')\ot \gamma _1(\gamma _2(u, u')_J, u'')_T\ot h'_J\gamma _3(u, u')\nonumber \\[2mm]
&&\;\;\;\;\;\;\;\;\;\;\;\;\;\;\;\;\ot \gamma _2(\gamma _2(u, u')_J, u'')\ot \gamma _3(\gamma _2(u, u')_J, u'')
h_T. \label{lrc15}
\end{eqnarray}
If we define on $U\ot H$ a multiplication, by 
\begin{eqnarray*}
&&(u\ot h)(u'\ot h')=\eta _1(\gamma _1(u, u')_T, \gamma _2(u, u')_J)\\[1mm]
&&\;\;\;\;\;\;\;\;\;\;\;\;\;\;\;\;\;\;\;\;\;\;\;\;\;\;\;\;\;\;\;\;\;\;\;\ot 
\eta _2(\gamma _1(u, u')_T, \gamma _2(u, u')_J)h_J\gamma _3(u, u')h'_T
\eta _3(\gamma _1(u, u')_T, \gamma _2(u, u')_J),
\end{eqnarray*}
for all $u, u'\in U$ and $h, h'\in H$, 
then this multiplication is associative and $1_U\ot 1_H$ is the unit. This algebra structure is called an 
{\em L-R-crossed product} and will be denoted by $U\nat _{J, T, \gamma , \eta }\;H$. We denote 
$u\nat h:=u\ot h$, for $u\in U$, $h\in H$. 
\end{theorem}
\begin{proof}
It is very easy to see that the relations (\ref{lrc1})-(\ref{lrc4}) imply that $1_U\nat 1_H$ is the unit. We check 
now the associativity (we denote by $T=t=\tau =\overline{\tau }=\overline{t}=\overline{T}$ and 
$J=j=i=I=\overline{J}=\overline{I}$ more copies of $T$ and $J$):\\[2mm]
${\;\;\;}$$[(u\nat h)(u'\nat h')](u''\nat h'')$
\begin{eqnarray*}
&=&[\eta _1(\gamma _1(u, u')_T, \gamma _2(u, u')_J)\\
&&\nat
\eta _2(\gamma _1(u, u')_T, \gamma _2(u, u')_J)h_J\gamma _3(u, u')h'_T
\eta _3(\gamma _1(u, u')_T, \gamma _2(u, u')_J)](u''\nat h'')\\
&=&\eta _1(\gamma _1(\eta _1(\gamma _1(u, u')_T, \gamma _2(u, u')_J), u'')_t, 
\gamma _2(\eta _1(\gamma _1(u, u')_T, \gamma _2(u, u')_J), u'')_j)\\
&&\nat \eta _2(\gamma _1(\eta _1(\gamma _1(u, u')_T, \gamma _2(u, u')_J), u'')_t, 
\gamma _2(\eta _1(\gamma _1(u, u')_T, \gamma _2(u, u')_J), u'')_j)\\
&&[\eta _2(\gamma _1(u, u')_T, \gamma _2(u, u')_J)h_J\gamma _3(u, u')h'_T
\eta _3(\gamma _1(u, u')_T, \gamma _2(u, u')_J)]_j\\
&&\gamma _3(\eta _1(\gamma _1(u, u')_T, \gamma _2(u, u')_J), u'')h''_t\\
&&\eta _3(\gamma _1(\eta _1(\gamma _1(u, u')_T, \gamma _2(u, u')_J), u'')_t, 
\gamma _2(\eta _1(\gamma _1(u, u')_T, \gamma _2(u, u')_J), u'')_j)\\
&\overset{(\ref{lrc5})}{=}&\eta _1(\gamma _1(\eta _1(\gamma _1(u, u')_T, \gamma _2(u, u')_J), u'')_t, 
\gamma _2(\eta _1(\gamma _1(u, u')_T, \gamma _2(u, u')_J), u'')_{j_{i_{I_{\overline{J}_{\overline{I}}}}}})\\
&&\nat \eta _2(\gamma _1(\eta _1(\gamma _1(u, u')_T, \gamma _2(u, u')_J), u'')_t, 
\gamma _2(\eta _1(\gamma _1(u, u')_T, \gamma _2(u, u')_J), u'')_{j_{i_{I_{\overline{J}_{\overline{I}}}}}})\\
&&\eta _2(\gamma _1(u, u')_T, \gamma _2(u, u')_J)_{\overline{I}}h_{J_{\overline{J}}}\gamma _3(u, u')_Ih'_{T_i}
\eta _3(\gamma _1(u, u')_T, \gamma _2(u, u')_J)_j\\
&&\gamma _3(\eta _1(\gamma _1(u, u')_T, \gamma _2(u, u')_J), u'')h''_t\\
&&\eta _3(\gamma _1(\eta _1(\gamma _1(u, u')_T, \gamma _2(u, u')_J), u'')_t, 
\gamma _2(\eta _1(\gamma _1(u, u')_T, \gamma _2(u, u')_J), u'')_{j_{i_{I_{\overline{J}_{\overline{I}}}}}})\\
&\overset{(\ref{lrc14})}{=}&\eta _1(\eta _1(\gamma _1(\gamma _1(u, u')_T, 
\gamma _2(\gamma _2(u, u')_J, u''))_{\tau }, 
\gamma _1(\gamma _2(u, u')_J, u''))_t, \\
&&\gamma _2(\gamma _1(u, u')_T, \gamma _2(\gamma _2(u, u')_J, u''))_{i_{I_{\overline{J}_{\overline{I}}}}})\\
&&\nat \eta _2(\eta _1(\gamma _1(\gamma _1(u, u')_T, 
\gamma _2(\gamma _2(u, u')_J, u''))_{\tau }, 
\gamma _1(\gamma _2(u, u')_J, u''))_t, \\
&&\gamma _2(\gamma _1(u, u')_T, \gamma _2(\gamma _2(u, u')_J, u''))_{i_{I_{\overline{J}_{\overline{I}}}}})\\
&&\eta _2(\gamma _1(\gamma _1(u, u')_T, \gamma _2(\gamma _2(u, u')_J, u''))_{\tau }, 
\gamma _1(\gamma _2(u, u')_J, u''))_{\overline{I}}h_{J_{\overline{J}}}\gamma _3(u, u')_Ih'_{T_i}\\
&&\gamma _3(\gamma _1(u, u')_T, 
\gamma _2(\gamma _2(u, u')_J, u''))\gamma _3(\gamma _2(u, u')_J, u'')_{\tau }\\
&&\eta _3(\gamma _1(\gamma _1(u, u')_T, 
\gamma _2(\gamma _2(u, u')_J, u''))_{\tau }, 
\gamma _1(\gamma _2(u, u')_J, u''))h''_t\\
&&\eta _3(\eta _1(\gamma _1(\gamma _1(u, u')_T, 
\gamma _2(\gamma _2(u, u')_J, u''))_{\tau }, 
\gamma _1(\gamma _2(u, u')_J, u''))_t, \\
&&\gamma _2(\gamma _1(u, u')_T, \gamma _2(\gamma _2(u, u')_J, u''))_{i_{I_{\overline{J}_{\overline{I}}}}})\\
&\overset{(\ref{lrc10})}{=}&\eta _1(\eta _1(\gamma _1(\gamma _1(u, u')_T, 
\gamma _2(\gamma _2(u, u')_J, u''))_{\tau _{\overline{\tau }}}, 
\gamma _1(\gamma _2(u, u')_J, u'')_t), \\
&&\gamma _2(\gamma _1(u, u')_T, \gamma _2(\gamma _2(u, u')_J, u''))_{i_{I_{\overline{J}_{\overline{I}}}}})\\
&&\nat \eta _2(\eta _1(\gamma _1(\gamma _1(u, u')_T, 
\gamma _2(\gamma _2(u, u')_J, u''))_{\tau _{\overline{\tau }}}, 
\gamma _1(\gamma _2(u, u')_J, u'')_t), \\
&&\gamma _2(\gamma _1(u, u')_T, \gamma _2(\gamma _2(u, u')_J, u''))_{i_{I_{\overline{J}_{\overline{I}}}}})\\
&&\eta _2(\gamma _1(\gamma _1(u, u')_T, \gamma _2(\gamma _2(u, u')_J, u''))_{\tau _{\overline{\tau }}}, 
\gamma _1(\gamma _2(u, u')_J, u'')_t)_{\overline{I}}h_{J_{\overline{J}}}\gamma _3(u, u')_Ih'_{T_i}\\
&&\gamma _3(\gamma _1(u, u')_T, 
\gamma _2(\gamma _2(u, u')_J, u''))\gamma _3(\gamma _2(u, u')_J, u'')_{\tau }h''_{t_{\overline{\tau }}}\\
&&\eta _3(\gamma _1(\gamma _1(u, u')_T, 
\gamma _2(\gamma _2(u, u')_J, u''))_{\tau _{\overline{\tau }}}, 
\gamma _1(\gamma _2(u, u')_J, u'')_t) \\
&&\eta _3(\eta _1(\gamma _1(\gamma _1(u, u')_T, 
\gamma _2(\gamma _2(u, u')_J, u''))_{\tau _{\overline{\tau }}}, 
\gamma _1(\gamma _2(u, u')_J, u'')_t), \\
&&\gamma _2(\gamma _1(u, u')_T, \gamma _2(\gamma _2(u, u')_J, u''))_{i_{I_{\overline{J}_{\overline{I}}}}})\\
&\overset{(\ref{lrc11})}{=}&\eta _1(\gamma _1(\gamma _1(u, u')_T, 
\gamma _2(\gamma _2(u, u')_J, u''))_{\tau _{\overline{\tau }_{\overline{T}}}}, \\
&&\eta _1(\gamma _1(\gamma _2(u, u')_J, u'')_t, 
\gamma _2(\gamma _1(u, u')_T, \gamma _2(\gamma _2(u, u')_J, u''))_{i_{I_{\overline{J}}}}))\\
&&\nat \eta _2(\gamma _1(\gamma _1(u, u')_T, 
\gamma _2(\gamma _2(u, u')_J, u''))_{\tau _{\overline{\tau }_{\overline{T}}}}, \\
&&\eta _1(\gamma _1(\gamma _2(u, u')_J, u'')_t, 
\gamma _2(\gamma _1(u, u')_T, \gamma _2(\gamma _2(u, u')_J, u''))_{i_{I_{\overline{J}}}}))\\
&&\eta _2(\gamma _1(\gamma _2(u, u')_J, u'')_t, 
\gamma _2(\gamma _1(u, u')_T, \gamma _2(\gamma _2(u, u')_J, u''))_{i_{I_{\overline{J}}}})\\
&&h_{J_{\overline{J}}}\gamma _3(u, u')_Ih'_{T_i}
\gamma _3(\gamma _1(u, u')_T, 
\gamma _2(\gamma _2(u, u')_J, u''))\gamma _3(\gamma _2(u, u')_J, u'')_{\tau }h''_{t_{\overline{\tau }}}\\
&&\eta _3(\gamma _1(\gamma _2(u, u')_J, u'')_t, 
\gamma _2(\gamma _1(u, u')_T, \gamma _2(\gamma _2(u, u')_J, u''))_{i_{I_{\overline{J}}}})_{\overline{T}}\\
&&\eta _3(\gamma _1(\gamma _1(u, u')_T, 
\gamma _2(\gamma _2(u, u')_J, u''))_{\tau _{\overline{\tau }_{\overline{T}}}}, \\
&&\eta _1(\gamma _1(\gamma _2(u, u')_J, u'')_t, 
\gamma _2(\gamma _1(u, u')_T, \gamma _2(\gamma _2(u, u')_J, u''))_{i_{I_{\overline{J}}}}))\\
&\overset{(\ref{lrc5}), \;(\ref{lrc6})}{=}&\eta _1(\gamma _1(\gamma _1(u, u')_T, 
\gamma _2(\gamma _2(u, u')_J, u''))_{\tau _{\overline{T}}}, \\
&&\eta _1(\gamma _1(\gamma _2(u, u')_J, u'')_t, 
\gamma _2(\gamma _1(u, u')_T, \gamma _2(\gamma _2(u, u')_J, u''))_{i_{I}}))\\
&&\nat \eta _2(\gamma _1(\gamma _1(u, u')_T, 
\gamma _2(\gamma _2(u, u')_J, u''))_{\tau _{\overline{T}}}, \\
&&\eta _1(\gamma _1(\gamma _2(u, u')_J, u'')_t, 
\gamma _2(\gamma _1(u, u')_T, \gamma _2(\gamma _2(u, u')_J, u''))_{i_{I}}))\\
&&\eta _2(\gamma _1(\gamma _2(u, u')_J, u'')_t, 
\gamma _2(\gamma _1(u, u')_T, \gamma _2(\gamma _2(u, u')_J, u''))_{i_{I}})\\
&&(h_{J}\gamma _3(u, u'))_Ih'_{T_i}
\gamma _3(\gamma _1(u, u')_T, 
\gamma _2(\gamma _2(u, u')_J, u''))(\gamma _3(\gamma _2(u, u')_J, u'')h''_t)_{\tau }\\
&&\eta _3(\gamma _1(\gamma _2(u, u')_J, u'')_t, 
\gamma _2(\gamma _1(u, u')_T, \gamma _2(\gamma _2(u, u')_J, u''))_{i_{I}})_{\overline{T}}\\
&&\eta _3(\gamma _1(\gamma _1(u, u')_T, 
\gamma _2(\gamma _2(u, u')_J, u''))_{\tau _{\overline{T}}}, \\
&&\eta _1(\gamma _1(\gamma _2(u, u')_J, u'')_t, 
\gamma _2(\gamma _1(u, u')_T, \gamma _2(\gamma _2(u, u')_J, u''))_{i_{I}})),
\end{eqnarray*}
${\;\;\;}$$(u\nat h)[(u'\nat h')(u''\nat h'')]$
\begin{eqnarray*}
&=&(u\nat h)[\eta _1(\gamma _1(u', u'')_T, \gamma _2(u', u'')_J)\\
&&\nat 
\eta _2(\gamma _1(u', u'')_T, \gamma _2(u', u'')_J)h'_J\gamma _3(u', u'')h''_T
\eta _3(\gamma _1(u', u'')_T, \gamma _2(u', u'')_J)]\\
&=&\eta _1(\gamma _1(u, \eta _1(\gamma _1(u', u'')_T, \gamma _2(u', u'')_J))_t, 
\gamma _2(u, \eta _1(\gamma _1(u', u'')_T, \gamma _2(u', u'')_J))_j)\\
&&\nat \eta _2(\gamma _1(u, \eta _1(\gamma _1(u', u'')_T, \gamma _2(u', u'')_J))_t, 
\gamma _2(u, \eta _1(\gamma _1(u', u'')_T, \gamma _2(u', u'')_J))_j)\\
&&h_j\gamma _3(u, \eta _1(\gamma _1(u', u'')_T, \gamma _2(u', u'')_J))\\
&&[\eta _2(\gamma _1(u', u'')_T, \gamma _2(u', u'')_J)h'_J\gamma _3(u', u'')h''_T
\eta _3(\gamma _1(u', u'')_T, \gamma _2(u', u'')_J)]_t\\
&&\eta _3(\gamma _1(u, \eta _1(\gamma _1(u', u'')_T, \gamma _2(u', u'')_J))_t, 
\gamma _2(u, \eta _1(\gamma _1(u', u'')_T, \gamma _2(u', u'')_J))_j)\\
&\overset{(\ref{lrc6})}{=}&
\eta _1(\gamma _1(u, \eta _1(\gamma _1(u', u'')_T, 
\gamma _2(u', u'')_J))_{t_{\overline{t}_{\tau _{\overline{\tau}_{\overline{T}}}}}}, 
\gamma _2(u, \eta _1(\gamma _1(u', u'')_T, \gamma _2(u', u'')_J))_j)\\
&&\nat \eta _2(\gamma _1(u, \eta _1(\gamma _1(u', u'')_T, \gamma _2(u', u'')_J))_{t_{\overline{t}_{\tau _{\overline{\tau}_{\overline{T}}}}}}, 
\gamma _2(u, \eta _1(\gamma _1(u', u'')_T, \gamma _2(u', u'')_J))_j)\\
&&h_j\gamma _3(u, \eta _1(\gamma _1(u', u'')_T, \gamma _2(u', u'')_J))\\
&&\eta _2(\gamma _1(u', u'')_T, \gamma _2(u', u'')_J)_th'_{J_{\overline{t}}}
\gamma _3(u', u'')_{\tau }h''_{T_{\overline{\tau }}}
\eta _3(\gamma _1(u', u'')_T, \gamma _2(u', u'')_J)_{\overline{T}}\\
&&\eta _3(\gamma _1(u, \eta _1(\gamma _1(u', u'')_T, \gamma _2(u', u'')_J))_{t_{\overline{t}_{\tau _{\overline{\tau}_{\overline{T}}}}}}, 
\gamma _2(u, \eta _1(\gamma _1(u', u'')_T, \gamma _2(u', u'')_J))_j)\\
&\overset{(\ref{lrc13})}{=}&
\eta _1(\gamma _1(\gamma _1(u, \gamma _1(u', u'')_T), 
\gamma _2(u', u'')_J)_{\overline{t}_{\tau _{\overline{\tau}_{\overline{T}}}}}, \\
&&\eta _1(\gamma _2(u, \gamma _1(u', u'')_T), \gamma _2(\gamma _1(u, \gamma _1(u', u'')_T), 
\gamma _2(u', u'')_J)_{\overline{J}})_j)\\
&&\nat \eta _2(\gamma _1(\gamma _1(u, \gamma _1(u', u'')_T), 
\gamma _2(u', u'')_J)_{\overline{t}_{\tau _{\overline{\tau}_{\overline{T}}}}}, \\
&&\eta _1(\gamma _2(u, \gamma _1(u', u'')_T), \gamma _2(\gamma _1(u, \gamma _1(u', u'')_T), 
\gamma _2(u', u'')_J)_{\overline{J}})_j)\\
&&h_j\eta _2(\gamma _2(u, \gamma _1(u', u'')_T), \gamma _2(\gamma _1(u, \gamma _1(u', u'')_T), 
\gamma _2(u', u'')_J)_{\overline{J}})\gamma _3(u, \gamma _1(u', u'')_T)_{\overline{J}}\\
&&\gamma _3(\gamma _1(u, \gamma _1(u', u'')_T), 
\gamma _2(u', u'')_J)h'_{J_{\overline{t}}}
\gamma _3(u', u'')_{\tau }h''_{T_{\overline{\tau }}}\\
&&\eta _3(\gamma _2(u, \gamma _1(u', u'')_T), \gamma _2(\gamma _1(u, \gamma _1(u', u'')_T), 
\gamma _2(u', u'')_J)_{\overline{J}})_{\overline{T}}\\
&&\eta _3(\gamma _1(\gamma _1(u, \gamma _1(u', u'')_T), 
\gamma _2(u', u'')_J)_{\overline{t}_{\tau _{\overline{\tau}_{\overline{T}}}}}, \\
&&\eta _1(\gamma _2(u, \gamma _1(u', u'')_T), \gamma _2(\gamma _1(u, \gamma _1(u', u'')_T), 
\gamma _2(u', u'')_J)_{\overline{J}})_j)\\
&\overset{(\ref{lrc12})}{=}&
\eta _1(\gamma _1(\gamma _1(u, \gamma _1(u', u'')_T)_{\overline{t}}, 
\gamma _2(u', u''))_{\tau _{\overline{\tau}_{\overline{T}}}}, \\
&&\eta _1(\gamma _2(u, \gamma _1(u', u'')_T), \gamma _2(\gamma _1(u, \gamma _1(u', u'')_T)_{\overline{t}}, 
\gamma _2(u', u''))_{J_{\overline{J}}})_j)\\
&&\nat \eta _2(\gamma _1(\gamma _1(u, \gamma _1(u', u'')_T)_{\overline{t}}, 
\gamma _2(u', u''))_{\tau _{\overline{\tau}_{\overline{T}}}}, \\
&&\eta _1(\gamma _2(u, \gamma _1(u', u'')_T), \gamma _2(\gamma _1(u, \gamma _1(u', u'')_T)_{\overline{t}}, 
\gamma _2(u', u''))_{J_{\overline{J}}})_j)\\
&&h_j\eta _2(\gamma _2(u, \gamma _1(u', u'')_T), \gamma _2(\gamma _1(u, \gamma _1(u', u'')_T)_{\overline{t}}, 
\gamma _2(u', u''))_{J_{\overline{J}}})\gamma _3(u, \gamma _1(u', u'')_T)_{\overline{J}}h'_{\overline{t}_J}\\
&&\gamma _3(\gamma _1(u, \gamma _1(u', u'')_T)_{\overline{t}}, 
\gamma _2(u', u''))
\gamma _3(u', u'')_{\tau }h''_{T_{\overline{\tau }}}\\
&&\eta _3(\gamma _2(u, \gamma _1(u', u'')_T), \gamma _2(\gamma _1(u, \gamma _1(u', u'')_T)_{\overline{t}}, 
\gamma _2(u', u''))_{J_{\overline{J}}})_{\overline{T}}\\
&&\eta _3(\gamma _1(\gamma _1(u, \gamma _1(u', u'')_T)_{\overline{t}}, 
\gamma _2(u', u''))_{\tau _{\overline{\tau}_{\overline{T}}}}, \\
&&\eta _1(\gamma _2(u, \gamma _1(u', u'')_T), \gamma _2(\gamma _1(u, \gamma _1(u', u'')_T)_{\overline{t}}, 
\gamma _2(u', u''))_{J_{\overline{J}}})_j)\\
&\overset{(\ref{lrc9})}{=}&
\eta _1(\gamma _1(\gamma _1(u, \gamma _1(u', u'')_T)_{\overline{t}}, 
\gamma _2(u', u''))_{\tau _{\overline{\tau}_{\overline{T}}}}, \\
&&\eta _1(\gamma _2(u, \gamma _1(u', u'')_T)_j, \gamma _2(\gamma _1(u, \gamma _1(u', u'')_T)_{\overline{t}}, 
\gamma _2(u', u''))_{J_{\overline{J}_{\overline{j}}}}))\\
&&\nat \eta _2(\gamma _1(\gamma _1(u, \gamma _1(u', u'')_T)_{\overline{t}}, 
\gamma _2(u', u''))_{\tau _{\overline{\tau}_{\overline{T}}}}, \\
&&\eta _1(\gamma _2(u, \gamma _1(u', u'')_T)_j, \gamma _2(\gamma _1(u, \gamma _1(u', u'')_T)_{\overline{t}}, 
\gamma _2(u', u''))_{J_{\overline{J}_{\overline{j}}}}))\\
&&\eta _2(\gamma _2(u, \gamma _1(u', u'')_T)_j, \gamma _2(\gamma _1(u, \gamma _1(u', u'')_T)_{\overline{t}}, 
\gamma _2(u', u''))_{J_{\overline{J}_{\overline{j}}}})\\
&&h_{j_{\overline{j}}}\gamma _3(u, \gamma _1(u', u'')_T)_{\overline{J}}h'_{\overline{t}_J}
\gamma _3(\gamma _1(u, \gamma _1(u', u'')_T)_{\overline{t}}, 
\gamma _2(u', u''))\gamma _3(u', u'')_{\tau }h''_{T_{\overline{\tau }}}\\
&&\eta _3(\gamma _2(u, \gamma _1(u', u'')_T)_j, \gamma _2(\gamma _1(u, \gamma _1(u', u'')_T)_{\overline{t}}, 
\gamma _2(u', u''))_{J_{\overline{J}_{\overline{j}}}})_{\overline{T}}\\
&&\eta _3(\gamma _1(\gamma _1(u, \gamma _1(u', u'')_T)_{\overline{t}}, 
\gamma _2(u', u''))_{\tau _{\overline{\tau}_{\overline{T}}}}, \\
&&\eta _1(\gamma _2(u, \gamma _1(u', u'')_T)_j, \gamma _2(\gamma _1(u, \gamma _1(u', u'')_T)_{\overline{t}}, 
\gamma _2(u', u''))_{J_{\overline{J}_{\overline{j}}}}))\\
&\overset{(\ref{lrc5}), \;(\ref{lrc6})}{=}&
\eta _1(\gamma _1(\gamma _1(u, \gamma _1(u', u'')_T)_{\overline{t}}, 
\gamma _2(u', u''))_{\tau _{\overline{T}}}, \\
&&\eta _1(\gamma _2(u, \gamma _1(u', u'')_T)_j, \gamma _2(\gamma _1(u, \gamma _1(u', u'')_T)_{\overline{t}}, 
\gamma _2(u', u''))_{J_{\overline{J}}}))\\
&&\nat \eta _2(\gamma _1(\gamma _1(u, \gamma _1(u', u'')_T)_{\overline{t}}, 
\gamma _2(u', u''))_{\tau _{\overline{T}}}, \\
&&\eta _1(\gamma _2(u, \gamma _1(u', u'')_T)_j, \gamma _2(\gamma _1(u, \gamma _1(u', u'')_T)_{\overline{t}}, 
\gamma _2(u', u''))_{J_{\overline{J}}}))\\
&&\eta _2(\gamma _2(u, \gamma _1(u', u'')_T)_j, \gamma _2(\gamma _1(u, \gamma _1(u', u'')_T)_{\overline{t}}, 
\gamma _2(u', u''))_{J_{\overline{J}}})\\
&&[h_{j}\gamma _3(u, \gamma _1(u', u'')_T)]_{\overline{J}}h'_{\overline{t}_J}
\gamma _3(\gamma _1(u, \gamma _1(u', u'')_T)_{\overline{t}}, 
\gamma _2(u', u''))[\gamma _3(u', u'')h''_{T}]_{\tau }\\
&&\eta _3(\gamma _2(u, \gamma _1(u', u'')_T)_j, \gamma _2(\gamma _1(u, \gamma _1(u', u'')_T)_{\overline{t}}, 
\gamma _2(u', u''))_{J_{\overline{J}}})_{\overline{T}}\\
&&\eta _3(\gamma _1(\gamma _1(u, \gamma _1(u', u'')_T)_{\overline{t}}, 
\gamma _2(u', u''))_{\tau _{\overline{T}}}, \\
&&\eta _1(\gamma _2(u, \gamma _1(u', u'')_T)_j, \gamma _2(\gamma _1(u, \gamma _1(u', u'')_T)_{\overline{t}}, 
\gamma _2(u', u''))_{J_{\overline{J}}}))\\
&\overset{(\ref{lrc15})}{=}&\eta _1(\gamma _1(\gamma _1(u, u')_{\overline{t}}, 
\gamma _2(\gamma _2(u, u')_j, u''))_{\tau _{\overline{T}}}, \\
&&\eta _1(\gamma _1(\gamma _2(u, u')_j, u'')_T, \gamma _2(\gamma _1(u, u')_{\overline{t}}, 
\gamma _2(\gamma _2(u, u')_j, u''))_{J_{\overline{J}}}))\\
&&\nat \eta _2(\gamma _1(\gamma _1(u, u')_{\overline{t}}, 
\gamma _2(\gamma _2(u, u')_j, u''))_{\tau _{\overline{T}}}, \\
&&\eta _1(\gamma _1(\gamma _2(u, u')_j, u'')_T, \gamma _2(\gamma _1(u, u')_{\overline{t}}, 
\gamma _2(\gamma _2(u, u')_j, u''))_{J_{\overline{J}}}))\\
&&\eta _2(\gamma _1(\gamma _2(u, u')_j, u'')_T, \gamma _2(\gamma _1(u, u')_{\overline{t}}, 
\gamma _2(\gamma _2(u, u')_j, u''))_{J_{\overline{J}}})\\
&&[h_j\gamma _3(u, u')]_{\overline{J}}h'_{\overline{t}_J}\gamma _3(\gamma _1(u, u')_{\overline{t}}, 
\gamma _2(\gamma _2(u, u')_j, u''))[\gamma _3(\gamma _2(u, u')_j, u'')h''_T]_{\tau }\\
&&\eta _3(\gamma _1(\gamma _2(u, u')_j, u'')_T, \gamma _2(\gamma _1(u, u')_{\overline{t}}, 
\gamma _2(\gamma _2(u, u')_j, u''))_{J_{\overline{J}}})_{\overline{T}}\\
&&\eta _3(\gamma _1(\gamma _1(u, u')_{\overline{t}}, 
\gamma _2(\gamma _2(u, u')_j, u''))_{\tau _{\overline{T}}}, \\
&&\eta _1(\gamma _1(\gamma _2(u, u')_j, u'')_T, \gamma _2(\gamma _1(u, u')_{\overline{t}}, 
\gamma _2(\gamma _2(u, u')_j, u''))_{J_{\overline{J}}}))
\end{eqnarray*}
and we can see that the two terms are equal. 
\end{proof}
\begin{remark}
If $U\nat _{J, T, \gamma , \eta }\;H$ is an L-R-crossed product, then we have 
\begin{eqnarray*}
&&(1_U\nat h)(1_U\nat h')=1_U\nat hh', \;\;\;
(u\nat h)(1_U\nat h')=u_T\nat hh'_T, \;\;\;
(1_U\nat h)(u'\nat 1_H)=u'_J\nat h_J, 
\end{eqnarray*}
for all $u, u'\in U$ and $h, h'\in H$.
\end{remark}
\section{Examples}\selabel{3}
\setcounter{equation}{0}
\begin{example}{\em
Let $H$ be a quasi-bialgebra and let $\mathcal{A}$ be an $H$-bimodule algebra, with notation as in \cite{pvo}. Define the linear maps
\begin{eqnarray*}
&&J:H\ot \mathcal{A}\rightarrow \mathcal{A}\ot H, \;\;\;J(h\ot \varphi )=h_1\cdot \varphi \ot h_2, \\
&&T:\mathcal{A}\ot H\rightarrow \mathcal{A}\ot H, \;\;\;T(\varphi \ot h)=\varphi \cdot h_2\ot h_1, \\
&&\gamma : \mathcal{A}\ot \mathcal{A}\rightarrow \mathcal{A}\ot \mathcal{A}\ot H, \;\;\;\gamma (\varphi \ot \varphi ')= \varphi \cdot z^3\ot z^1\cdot \varphi '\ot z^2,\\
&&\eta : \mathcal{A}\ot \mathcal{A}\rightarrow \mathcal{A}\ot H\ot H, \;\;\;\eta (\varphi \ot \varphi ')=
(x^1\cdot \varphi \cdot y^2)(x^2\cdot \varphi '\cdot y^3)\ot x^3\ot y^1, 
\end{eqnarray*}
for all $h\in H$ and $\varphi , \varphi '\in \mathcal{A}$. Then one can check that these maps satisfy the conditions (\ref{lrc1})-(\ref{lrc15}), so we can consider the L-R-crossed product $\mathcal{A}\nat _{J, T, \gamma , \eta }\;H$. One can check that its multiplication boils down to 
$$(\varphi \ot h)(\varphi '\ot h')=(x^1\cdot \varphi \cdot z^3h'_2y^2)(x^2h_1z^1\cdot \varphi '\cdot y^3)\ot x^3h_2z^2h'_1y^1,$$
which is exactly the multiplication of the L-R-smash product $\mathcal{A}\nat H$ introduced in \cite{pvo}.
} 
\end{example}
\begin{example}{\em
Let $A\; _Q\otimes _RB$ be an L-R-twisted tensor product between the associative unital algebras $A$ and $B$ as in \cite{ciungu}. Define the linear maps 
\begin{eqnarray*}
&&J:B\ot A\rightarrow A\ot B, \;\;\;J=R,\\
&&T:A\ot B\rightarrow A\ot B, \;\;\;T=Q, \\
&&\gamma :A\ot A\rightarrow A\ot A\ot B, \;\;\;\gamma (a\ot a')= a\ot a'\ot 1_B,\\
&&\eta : A\ot A\rightarrow A\ot B\ot B, \;\;\;\eta (a\ot a')=aa'\ot 1_B\ot 1_B, 
\end{eqnarray*}
for all $a, a'\in A$. Then one can check that these maps satisfy the conditions (\ref{lrc1})-(\ref{lrc15}), so we can consider the L-R-crossed product $A\nat _{J, T, \gamma , \eta }\;B$. One can check that its multiplication reduces to 
$$(a\ot b)(a'\ot b')=a_Qa'_R\ot b_Rb'_Q,$$
which is exactly the multiplication of $A\; _Q\otimes _RB$.}
\end{example}
\begin{example}{\em
Let $W\overline{\otimes}_{P, \nu }B$ be a mirror version of a Brzezi\'{n}ski crossed product, as in  \cite{iterated}. Define the linear maps 
\begin{eqnarray*}
&&J:B\ot W\rightarrow W\ot B, \;\;\;J=P,\\
&&T:W\ot B\rightarrow W\ot B, \;\;\;T=id, \\
&&\gamma :W\ot W\rightarrow W\ot W\ot B, \;\;\;\gamma (w\ot w')= w\ot w'\ot 1_B,\\
&&\eta : W\ot W\rightarrow W\ot B\ot B, \;\;\;\eta (w\ot w')=\nu (w\ot w')\ot 1_B, 
\end{eqnarray*}
for all $w, w'\in W$. Then one can check that these maps satisfy the conditions (\ref{lrc1})-(\ref{lrc15}), so we can consider the L-R-crossed product $W\nat _{J, T, \gamma , \eta }\;B$. One can check that its multiplication becomes 
$$(w\ot b)(w'\ot b')=\nu _1(w, w'_P)\ot \nu _2(w, w'_P)b_Pb',$$
which is exactly the multiplication of $W\overline{\otimes}_{P, \nu }B$.  }
\end{example}
\begin{example}{\em
Let $H$ be an associative unital algebra, let $V$ and $W$ be two linear spaces with distinguished elements $1_V\in V$ and $1_W\in W$, and assume that we have a Brzezi\'{n}ski crossed product $H\ot _{R, \sigma }V$ and a mirror version Brzezi\'{n}ski crossed product  $W\overline{\otimes}_{P, \nu }H$, with notation as in \cite{iterated}. Assume moreover that we have a linear map $Q:V\ot W\rightarrow W\ot H\ot V$ satisfying the conditions in Theorem 2.3 in \cite{iterated}, therefore the two crossed products 
$W\overline{\otimes}_{P, \nu }H$ and $H\ot _{R, \sigma }V$ can be iterated, so we have an associative algebra structure on 
$W\ot H\ot V$. 

Consider now the linear space $U=W\ot V$, with distinguished element $1_U=1_W\ot 1_V$, and define the linear maps 
\begin{eqnarray*}
&&J:H\ot (W\ot V)\rightarrow (W\ot V)\ot H, \;\;\;J(h\ot (w\ot v))=(w_P\ot v)\ot h_P, \\
&&T:(W\ot V)\ot H\rightarrow (W\ot V)\ot H, \;\;\;T((w\ot v)\ot h)=(w\ot v_R)\ot h_R, \\
&&\gamma :(W\ot V)\ot (W\ot V)\rightarrow (W\ot V)\ot (W\ot V)\ot H, \\
&&\;\;\;\;\;\;\;\;\gamma ((w\ot v)\ot (w'\ot v'))= (w\ot Q_V(v, w'))\ot (Q_W(v, w')\ot v')\ot Q_H(v, w'), \\
&&\eta : (W\ot V)\ot (W\ot V)\rightarrow (W\ot V)\ot H\ot H, \\
&&\;\;\;\;\;\;\;\;\eta ((w\ot v)\ot (w'\ot v'))=(\nu _1(w, w')\ot \sigma _2(v, v'))\ot \nu_2(w, w')\ot \sigma _1(v, v'), 
\end{eqnarray*}
for all $v, v'\in V$ and $w, w'\in W$. Then one can check that these maps satisfy the conditions (\ref{lrc1})-(\ref{lrc15}), so we can consider the L-R-crossed product $(W\ot V)\nat _{J, T, \gamma , \eta }\;H$. One can check that its multiplication becomes 
\begin{eqnarray*}
((w\ot v)\ot h)((w'\ot v')\ot h')&=&\nu _1(w, Q_W(v, w')_P)\ot  \sigma _2(Q_V(v, w')_R, v')\ot \\
&&\;\;\;\;\;\nu _2(w, Q_W(v, w')_P)h_PQ_H(v, w')h'_R\sigma _1(Q_V(v, w')_R, v'), 
\end{eqnarray*}
and, up to the canonical linear isomorphism $W\ot V\ot H\simeq W\ot H\ot V$, this coincides with the multiplication of the iterated crossed product structure on $W\ot H\ot V$ (see the proof of (iii) in Theorem 2.3 in \cite{iterated}).  }
\end{example}

\begin{center}
ACKNOWLEDGEMENTS
\end{center}
The author was partially supported by a grant from UEFISCDI,
project number PN-III-P4-PCE-2021-0282.


\end{document}